\documentclass[12pt]{amsart}

\usepackage{amsfonts}
\usepackage{amsmath}
\usepackage{amssymb}
\usepackage{amsthm}
\usepackage{enumerate}
\usepackage{geometry}
\usepackage{enumitem}
\usepackage{indentfirst}
\usepackage{etoolbox}
\usepackage{hyperref}
\DeclareMathOperator{\supp}{supp}

\DeclareMathOperator{\dom}{dom}

\DeclareMathOperator{\ordem}{o}

\theoremstyle{definition}
\newtheorem{defin}{Definition}[section]
\newtheorem{prop}[defin]{Proposition}
\newtheorem{lem}[defin]{Lemma}
\newtheorem{teo}[defin]{Theorem}
\newtheorem{cor}[defin]{Corollary}

\newtheorem{prob}[defin]{Problem}

\AtEndEnvironment{defin}{\null\hfill\qedsymbol}%

\title[Forcing countably compact group topologies with convergent sequences]{Forcing a classification of non-torsion Abelian groups  of size at most $2^\mathfrak c$ with non-trivial convergent sequences}

	\author[M. K. Bellini]{Matheus Koveroff Bellini}
	\author[V. O. Rodrigues]{Vinicius de Oliveira Rodrigues}
	\author[A. H. Tomita]{Artur Hideyuki Tomita}
	\address{Depto de Matem\'atica, Instituto de Matem\'atica e Estat\'istica, Universidade de S\~ao Paulo, Rua do Mat\~ao, 1010 -- CEP 05508-090, S\~ao Paulo, SP - Brazil}
	\email{tomita@ime.usp.br, matheusb@ime.usp.br, vinior@ime.usp.br}
	\thanks{The first listed author has received financial support from FAPESP 2017/15709-6.}
	\thanks{The second listed author has received financial support from FAPESP 2017/15502-2.}
	\thanks{The third listed author has received financial support from FAPESP 2016/26216-8.}


\begin{document}

\maketitle

\begin{abstract} We force a classification of all the Abelian groups of cardinality at most $2^\mathfrak c$ that admit a countably compact group with a non-trivial convergent sequence. In particular, we answer (consistently) Question 24 of Dikranjan and Shakhmatov \cite{dikranjan&shakhmatov3} for cardinality at most $2^{\mathfrak c}$, by showing that if a non-torsion Abelian group of size at most $2^\mathfrak c$ admits a countably compact Hausdorff group topology, then it admits a countably compact Hausdorff group topology with non-trivial convergent sequences.
\end{abstract}

\section{Introduction}

\subsection{Some history}

There are three natural questions concerning countably compact Abelian group without non-trivial convergent sequences that can be asked separately or jointly.

1) What groups admit such topologies?

2) How large they can be?

3) Do they exist in ZFC?

Question 1  was solved under Martin's Axiom in \cite{dikranjan&tkachenko} for Abelian groups of cardinality $\mathfrak c$. This was later improved in \cite{bellini&boero&rodrigues&tomita} under the use of $\mathfrak c$ selective ultrafilters.

Dikranjan and Shakhmatov \cite{dikranjan&shakhmatov2} used a forcing model to classify all Abelian groups of cardinality at most $2^{\mathfrak c}$ that admit a countably compact group topology (without non-trivial convergent sequences).

Question 2 was solved for torsion groups in Castro-Pereira and Tomita \cite{castro-pereira&tomita2010}. They  classified, using some cardinal arithmetic and the existence of a selective ultrafilter $p$, all the torsion groups that admit a $p$-compact group topology (without non-trivial convergent sequences).  This gave the first arbitrarily large countably compact groups without non-trivial convergent sequences. Earlier examples had their size limited to $2^\mathfrak c$

Question 3 is the best known question in the subject. It has been finally answered by M. Hrusak, U. A. Ramos-Garcia, J. van Mill and S. Shelah in \cite{michaelnew}.  The main new ingredient in the ZFC construction is the use of a clever filter which takes care of the combinatorics that guarantee the existence of accumulation points. This new idea has yet two limitations, the construction depends on the use of a group of finite order and the example has cardinality $\mathfrak c$. It is not yet known if the example could be improved to obtain an example of cardinality strictly greater than $\mathfrak c$. These questions appear in their article:

\begin{prob} Are there larger countably compact groups without non-trivial convergent sequences in ZFC?
 \end{prob}
 
 \begin{prob} Is there a torsion-free countably compact group without non-trivial convergent sequences in ZFC?
 \end{prob}
 
 Due to this last question, we are still far from a classification (in ZFC) of countably compact Abelian groups with no convergent sequences of cardinality $ \mathfrak c$. And due to the first question, we are even farther from a classification of group of cardinality $\leq 2^\mathfrak c$ in ZFC. 

To these questions, one can add an interesting question from 
Dikranjan and Shakhmatov \cite{dikranjan&shakhmatov3}, Question 24 whether an infinite group admitting a countably compact Hausdorff group topology can be endowed with a countably compact Hausdorff group topology that contains a non-trivial convergent sequence.
 It is well-known the difficult to construct countably compact groups without nontrivial convergent sequences, but once there was progress to classify them, it became natural to 
ask how hard would it be to add a convergent sequence. They also asked a similar question for pseudocompact groups, which was answered by
Galindo, Garcia-Ferreira and Tomita \cite{galindo&garcia-ferreira&tomita}.

In \cite{galindo&garcia-ferreira&tomita}, it was also noted that, in ZFC, if a torsion Abelian group admits a countably compact group topology then it admits a countably compact group topology with an infinite compact metric subgroup, so, in particular, with a non-trivial convergent sequence. Thus, the real difficult for Dikranjan and Shakhmatov's question is about non-torsion groups.

Using the example in \cite{castro-pereira&tomita2010}, it is consistent that every torsion group admits a countably compact group topology if and only if admits a countably compact group topology without non-trivial convergent sequences if and only if it admits a countably compact group topology with a non-trivial convergent sequence.

Boero, Garcia-Ferreira and Tomita \cite{boero&garcia-ferreira&tomita} showed that the existence of ${\mathfrak c}$ selective ultrafilters implies that the free Abelian group of size $\mathfrak{c}$ admits a  group topology that makes it countably compact with a non-trivial convergent sequence. Bellini,  Boero, Castro-Pereira, Rodrigues and Tomita \cite{bellini&boero&castro&rodrigues&tomita} showed from $\mathfrak p = \mathfrak c$ that every non-torsion Abelian group of cardinality $\mathfrak c$ that admits a countably compact group topology also admits one which in addition has a convergent sequence. Bellini, Boero, Rodrigues and Tomita \cite{bellini&boero&rodrigues&tomita} showed from the existence of $2^\mathfrak c$ selective ultrafilters that every Abelian group of cardinality $\mathfrak c$ as above admits $2^\mathfrak c$ non-homeormorphic such topologies with or without convergent sequences and  any given weight of cardinality between $\mathfrak c$ and $2^\mathfrak c$. It also provided a countably compact group topology, one with and other without non-trivial convergent sequences,  for some (but not all) non-torsion groups of cardinality $2^\mathfrak c$.

In their forcing model, Dikranjan and Shakhmatov \cite{dikranjan&shakhmatov2}
showed a non-torsion Abelian group of cardinality at most $2^\mathfrak c$ admits a countably compact group topology (without non-trivial convergent sequences) if and only if the free rank of $G$ is $\geq\mathfrak{c}$ and, for all $d, n \in \mathbb{N}$ with $d \mid n$, the group $dG[n]$ is either finite or has cardinality $\geq \mathfrak{c}$. Forcing is used to prove the `only if' of the equivalence, whereas the other implication holds in ZFC.

We will use here the ZFC implication. In the forcing examples to produce groups without non-trivial convergent sequences it is used some ideas closely related to the construction of HFD groups. We have to proceed differently to make a sequence converge.

\subsection{Basic results, notation and terminology}

We recall that a $T_1$ topological space is \emph{countably compact} if every infinite subspace has an accumulation point in the space.

The following definition was introduced in \cite{bernstein} and is closely related to countable compactness.

\begin{defin}\label{def_p-limit}
Let $p$ be a free ultrafilter on $\omega$ and let $s:\omega\rightarrow X$ be a sequence in a topological space $X$. We say that $x \in X$ is a \emph{$p$-limit point} of $s$ if, for every neighborhood $U$ of $s$, $\{n \in \omega : s(n) \in U\} \in p$.

We say a topological space is $p$-compact if every sequence into it has a $p$-limit point.
\end{defin}

 If $X$ is a Hausdorff space a sequence $s$ has at most one $p$-limit point $x$ and we write $x = p$-$\lim s$.

The set of all free ultrafilters on $\omega$ will be denoted by $\omega^{*}$. It is not difficult to show that a $T_1$ topological space $X$ is countably compact if and only if, each sequence in $X$ there exists $p \in \omega^*$ such that $s$ has a $p$-limit point. There are several similarities between the theories of limits of sequences and of $p$-limits.

\begin{prop}\label{prop_p-limit_product}
If $p \in \omega^{*}$ and $(X_i : i \in I)$ is a family of topological spaces, then $(y_{i})_{i \in I} \in \prod_{i \in I} X_i$ is a $p$-limit point of a sequence $((x_{i}^{n})_{i \in I} : n \in \omega)$ in $\prod_{i \in I} X_i$ if, and only if, $y_i = p-\lim (x_{i}^{n} : n \in \omega)$ for every $i \in I$.\qed
\end{prop}

The following proposition is straightforward to prove. A proof can be found in \cite{hindman2011algebra} (Theorem 3.54).

\begin{prop}
	Let $X$, $Y$ be topological spaces and $f:X\rightarrow Y$ be a continuous function, $s:\omega\rightarrow X$ be a sequence in $X$ and $p \in \omega^*$. It follows that if $x=p$-$\lim (s_n: n \in \omega)$, then $f(x)=p$-$\lim (f(s_n): n \in \omega)$.\qed
\end{prop}
Since $+$ and $-$ are continuous functions in topological groups, it follows from the two previous propositions that:
\begin{prop}\label{prop_p-limit}
Let $G$ be a topological group and $p \in \omega^{*}$.

\begin{enumerate}

  \item If $(x_n : n \in \omega)$ and $(y_n : n \in \omega)$ are sequences in $G$ and $x, y \in G$ are such that $x = p-\lim(x_n : n \in \omega)$ and $y = p-\lim(y_n : n \in \omega)$, then $x + y = p-\lim (x_n + y_n : n \in \omega)$;

  \item If $(x_n : n \in \omega)$ is a sequence in $G$ and $x \in G$ is such that $x = p-\lim(x_n : n \in \omega)$, then $- x = p-\lim(- x_n : n \in \omega)$.\qed

\end{enumerate}
\end{prop}

A \emph{pseudointersection} of a family $\mathcal{G}$ of sets is an infinite set that is almost contained in every member of $\mathcal{G}$. We say that a family $\mathcal{G}$ of infinite sets has the \emph{strong finite intersection property} (SFIP, for short) if every finite subfamily of $\mathcal{G}$ has infinite intersection. 

We denote the set of positive natural numbers by $\mathbb{N}$, the integers by $\mathbb{Z}$, the rationals by $\mathbb{Q}$ and the reals by $\mathbb{R}$. The unit circle group $\mathbb{T}$ will be identified with the metric group $(\mathbb{R} / \mathbb{Z}, \delta)$ where $\delta$ is given by $\delta(x + \mathbb{Z}, y + \mathbb{Z}) = \min\{|x - y + a| : a \in \mathbb{Z}\}$ for every $x, y \in \mathbb{R}$. Given a subset $A$ of $\mathbb{T}$, we will denote by $\delta(A)$ the diameter of $A$ with respect to the metric $\delta$. The set of all non-empty open arcs of $\mathbb{T}$ will be denoted by $\mathcal{B}$.

Let $X$ be a set and $(G, +, 0)$ be a group. We denote by $G^{X}$ the product $\prod_{x \in X} G_x$ where $G_x = G$ for every $x \in X$. The \emph{support} of $g \in G^{X}$ is the set $\{x \in X : g(x) \neq 0\}$, which will be designated as $\supp g$. The set $\{g \in G^{X} : |\supp g| < \omega\}$ will be denoted by $G^{(X)}$. If $f:\omega\rightarrow G^{(X)}$ is a sequence, then $\supp f=\bigcup_{n \in \omega}\supp f(n)$.

The torsion part $T(G)$ of an Abelian group $G$ is the set $\{x \in G : nx = 0 \ \hbox{for some} \ n \in \mathbb{N}\}$. Clearly, $T(G)$ is a subgroup of $G$. For every $n \in \mathbb{N}$, we put $G[n] = \{x \in G : nx = 0\}$. In the case $G = G[n]$, we say that $G$ is \emph{of exponent $n$} provided that $n$ is the minimal positive integer with this property. The order of an element $x \in G$ will be denoted by $\ordem(x)$.

A non-empty subset $S$ of an Abelian group $G$ is said to be \emph{independent} if $0 \not \in S$ and, given distinct elements $s_1, \ldots, s_n$ of $S$ and integers $m_1, \ldots, m_n$, the relation $m_1 s_1 + \ldots + m_n s_n = 0$ implies that $m_i s_i = 0$, for all $i$. The free rank $r(G)$ of $G$ is the cardinality of a maximal independent subset of $G$ such that all of its elements have infinite order. It is easy to verify that $r(G) = |G / T(G)|$ if $r(G)$ is infinite.

An Abelian group $G$ is called \emph{divisible} if, for each $g \in G$ and each $n \in \mathbb{N} \setminus \{0\}$, there exists $x \in G$ such that $nx = g$. If $n \in \mathbb{N}$, we denote by $G[n]$ the set $\{x \in G : nx = 0\}$.

The proof of the next three results are well known basic results of the theory of divisible groups can be found in \cite{robinson}.

\begin{prop}
Let $G$ be an Abelian group, $H$ be a subgroup of $G$, $\tilde{G}$ be a divisible group and $f: H \to \tilde{G}$ be a group homomorphism. There exists a group homomorphism $F: G \to \tilde{G}$ such that $F \hspace*{-0.1cm} \upharpoonright_{H} = f$.\qed
\end{prop}

The group $\mathbb{Q} / \mathbb{Z}$  is called the \emph{quasicyclic group}. 

\begin{teo}\label{teo_classificacao_grupos_divisiveis}
An Abelian group is divisible if and only if, it is isomorphic to a direct sum of copies of $\mathbb{Q}$ and of quasicyclic groups.\qed
\end{teo}

\begin{teo}\label{teo_imersao_num_grupo_divisivel}
Every Abelian group is isomorphic to a subgroup of a divisible group.
\end{teo}

\section{The groups for the immersion}

We present some of the notation that will be used throughout this article.

We fix a partition $\{P_0, P_1\}$ of $\mathfrak c$ such that $|P_0|=|P_1|=\mathfrak c$, and $\omega+1\subseteq P_1$.
Let $\{R_0,R_1\}$ be a partition of $2^\mathfrak c \setminus \mathfrak c$ such that $|R_0|=|R_1|=2^\mathfrak c$. Define 
$\mathbb U= \mathbb{Q}^{(R_0)}\oplus\mathbb{Q}^{(R_1)}$ and $\mathbf{U}= (\mathbb{Q}/\mathbb{Z})^{(R_0)}\oplus\mathbb{Q}^{(R_1)}$.

We define $\mathbb W = (\mathbb{Q} /\mathbb{Z})^{( P_0 \times \omega)} \oplus \mathbb{Q}^{ (P_1)} \oplus \mathbf{U}$. We also let $\mathbb X= \mathbb{Q} ^{( P_0\times \omega)} \oplus \mathbb{Q}^{ (P_1)} \oplus \mathbb{U}$.

Throughout this work, the main group will be $\mathbf X = \oplus_{n >1,m \geq 1 } (\mathbb Q/\mathbb Z)^{(C_{n,m} \times m)}\oplus \mathbb Q^{(P_1)}\oplus \mathbf{U}$, where $\{ C_{n,m}:\, n>1, m\geq 1\}$ is a partition of $P_0$ into pieces of cardinality $\mathfrak c$ and $C_0=\bigcup_{m,n>1}C_{n,m}$ is such that $P_0\setminus C_0=\bigcup_{n>1}C_{n,1}$ has cardinality $\mathfrak c$. We will also define  $C_1 \subseteq P_1$ with $|C_1|=\mathfrak c$, $\omega\subseteq C_1$ and $|P_1 \setminus C_1|=\mathfrak c$. 
To simplify the notation of the forcing later we will also
partition $C_1$ as $\{C_{1,m}:\, m>1\}$ and let $C_{1,1}=P_1 \setminus C_1$.

\subsection{Structure of the article}
 We use forcing to construct an injective group homomorphism $\Phi :\, {\mathbf X}\to \mathbb{T}^{\mathfrak c}$. The range of this homomorphism will be countably compact and have convergent sequences. Each forcing condition will be a partial countable piece of this homomorphism, and the existence of suitable conditions was proved in \cite{bellini&boero&castro&rodrigues&tomita}.
 
Of course, not every subgroup of $X$ is countably compact. However, we show that if $H$ is a group such that $2^\mathfrak c \geq |H| = r(H) = \mathfrak{c}$ and for all $d, n \in \mathbb{N}$ with $d \mid n$, the group $dH[n]$ is either finite or has cardinality at least $\mathfrak{c}$, then it is isomorphic to a subgroup of $\mathbf X$ that is countably compact and has convergent sequences considering the subspace topology of $X$ generated by $\Phi$. To show that such a copy exists, we define the concept of \textit{large subgroup of } $X$, which was inspired by the concept of \textit{nice immersion} we defined in \cite{bellini&boero&castro&rodrigues&tomita}. This concludes the classification.

The convergent sequences will be the sequences $(\frac{n!}{S} \chi_n:\, n \in \omega )$ in $G\subseteq {\mathbf X}$ (identifying $G$ with its copy in ${\mathbf X}$), for each positive integer $S$ and they will converge to $0$. 

\subsection{More notation}

Given $w \in \mathbb{W}$ or $w \in \mathbb X$, $x \in (P_0 \times \omega) \cup R_0$ and $y \in  P_1 \cup R_1$, we denote by $w(x)$ the $x$-th coordinate of $w$ and $w(y)$ the $y$-th of $w$, so the functions $w\rightarrow w(x)$ and $w\rightarrow w(y)$ are the natural projections.

We also fix well defined numerators and denominators for fractions: if $r \in  \mathbb Q/\mathbb Z$, then $p(r)$ and $q(r)$ are the unique integers $p, q$ such that $q>0$, $\gcd(p, q)=1$, $0\leq p<q$ and $r=\frac{p}{q}+\mathbb Z$. Likewise, if $r \in \mathbb Q$,  $p(r)$ and $q(r)$ are the unique integers $p, q$ such that $q>0$, $\gcd(p, q)=1$ and $r=\frac{p}{q}$.

Given $w \in \mathbb{W}$ (or $w \in \mathbb X$), we denote by $w^0$ and $w^1$ the unique elements of $\mathbb{W}$ (or $\mathbb X$) such that $\supp w^0 \subseteq ((P_0 \times \omega) \cup R_0)$, $\supp w^1 \subseteq P_1 \cup R_1$ and $w=w^0+w^1$, that is, $w\rightarrow w^0$ is the natural projections into $(\mathbb Q/\mathbb Z)^{((P_0 \times \omega) \cup R_0)}$ (or $\mathbb Q^{((P_0 \times \omega )\cup R_0)}$) and $w \rightarrow w^1$ is the natural projection into $\mathbb Q^{(P_1\cup R_1)}$. Also, we call $w^{1,0}$ and $w^{1,1}$ the natural projections of $w$ onto $\mathbb Q^{(\omega)}$ and $\mathbb Q^{((P_1 \cup R_1)\setminus \omega)}$, respectively.

We also define $p(w) = \max\{|p(w(z))| : z \in \supp w\}$  and $q(w) = \max\{q(w(z)) : z \in \supp w \}$ if $w \neq 0$. We define $p(0)=0$ and $q(0)=0$.

Similarly, given $g:\omega\to \mathbb{W} $ (or $\mathbb X$), we define $g^0$, $g^{1}$, $g^{1, 0}$ and $g^{1, 1}$. So $g=g^0+g^1=g^0+g^{1, 0}+g^{1, 1}$, $\supp g^0\subseteq ((P_0\times \omega) \cup R_0)$, $\supp g^1\subseteq P_1 \cup R_1$, $\supp g^{1, 0}\subseteq \omega$ and $\supp g^{1, 1}\subseteq (P_1\cup R_1)\setminus \omega$; where  $\supp h = \bigcup\{ \supp h(k):\, k\in \omega\}$  for a sequence $h$.

It will be useful to be able to easily transform an element of $\mathbb X$ into an element of $\mathbb{W} $. Thus, given $w \in \mathbb X$, we define $[w]$ as the unique element of $\mathbb W$ such that for every $x \in (P_0\times \omega) \cup (R_0)$, $[w](x)=w(x)+\mathbb Z$ and for every $y \in P_1 \cup R_1$, $[w](y)=w(y)$. Clearly, the function $w\rightarrow [w]$ is a group homomorphism from $\mathbb X$ onto $\mathbb{W}$. Given a function $g:\omega\rightarrow \mathbb X$, we also define $[g]:\omega\rightarrow \mathbb{W}$ be given by $[g](n)=[g(n)]$ for every $n \in \omega$.

\section{Types of sequences}

\subsection{Associating sequences to a type}
We will make some minor adjustments to the definition of the types defined in \cite{bellini&boero&castro&rodrigues&tomita} to be able to work with larger groups. Basically, we define all the types on $\mathbb W$ as in \cite{bellini&boero&castro&rodrigues&tomita}. However, the group $\mathbb W$ used in \cite{bellini&boero&castro&rodrigues&tomita} is not the same group $\mathbb W$ we are using in this article, since there $\mathbb W$ does not have the component $\mathbf U$.

In this section we define the 11 types of sequences related to a a subgroup $G$ of $\mathbb W$ and state the theorem that every sequence is related to one of them.

The name of the sequences which are of one of these 11 types of sequences for $G$ (which will be defined in the following subsections) by $\mathcal H_G$.

The main result we will state in this section is the following, which, in particular, implies that when working with the existence of accumulation points for a sequence, by passing to a subsequence it is enough to guarantee the existence of an accumulation point for the 11 types and the convergence of the sequence of $(\frac{n!}{S} \chi_n:\, n \in \omega)$ to $0$ for each positive integer $S$. The proof is the same as the proof of Theorem 3.1. of \cite{bellini&boero&castro&rodrigues&tomita} and is omitted, although the $\mathbb W$ is different.

\begin{teo}\label{prop_subseq}
Let $f : \omega \to \mathbb W$ is given by $f(n) =n!\chi_n$ for every $n \in \omega$. Let $G$ be a subgroup of $\mathbb W$ containing $(0, \chi_n)$ for every $n \in \omega$. Let $g: \omega \to \mathbb X$ with $[g] \in G^\omega$. 

Then there exists $h:\omega\rightarrow \mathbb X$ such that $h\in \mathcal{H}_G$ or $[h]$ is constant and in $G$, $c\in \mathbb X$ with $[c] \in G$, $F \in [\omega]^{< \omega}$, $p_i, q_i \in \mathbb{Z}$ with $q_i \neq 0$ for every $i \in F$, $(j_{i}:i\in F)$ increasing enumerations of subsets of $\omega$  and $j: \omega \to \omega$ strictly increasing such that

\[g \circ j = h + c + \sum_{i \in F} \frac{p_i}{q_i}   f\circ j_{i}\]

with $q_i\leq j_i(n)$ for each $n \in \omega$ and $i \in F$ (which implies $[\frac{1}{q_i} f\circ j_i] \in G^\omega$ since $q_i|((j_i(n))!)$ for each $i \in F$ and $n \in \omega$).\qed
\end{teo}

\subsection{ The types} Each type uses some point of the support and the supports are divided in three groups.

\begin{defin}[The types related to $(R_1 \cup P_1)\setminus \omega$] Let $G$ be a subgroup of $\mathbb W$. We define the first three types of sequence (with respect to $G$) as follows:\label{def_tipos_P1}
Let $g : \omega \to \mathbb X$ be such that $[g(n)]\in G$, for every $n \in \omega$. \\

We say that $g$ is \emph{of type 1} if $\supp g^{1,1}(n) \setminus \cup_{m < n} \supp g^{1,1}(m) \neq \emptyset$, for every $n \in \omega$.\\

We say that $g$ is \emph{of type 2} if $q(g^{1,1}(n)) > n$, for every $n \in \omega$.\\

We say that $g$ is \emph{of type 3} if $\{q(g^{1,1}(n)) : n \in \omega\}$ is bounded and $|p(g^{1,1}(n))| > n$, for every $n \in \omega$.
\end{defin}

\begin{defin}[The types related to $\omega$]
Let $G\subseteq \mathbb W$ be a subgroup such that $\{[\chi_n]:\, n \in \omega\} \subseteq G$. Let $g : \omega \to \mathbb X$ be such that $[g(n)] \in G$, for every $n \in \omega$. Then we define types 4 to 9 (with respect to $G$) as follows: \\

We say that $g$ is \emph{of type 4} if $q(g(n)) > n$, for every $n \in \omega$.\\

We say that $g$ is \emph{of type 5} if there exists $M \in \bigcap_{n \in \omega} \supp g^{1,0}(n)$ such that $\{q(g(n)) : n \in \omega\}$ is bounded and $|p(g(n)(M))| > n$, for every $n \in \omega$.\\

To define types 6, 7 and $8$, suppose $g$ is  such that for each $n \in \omega$, there exists $M_n \in \supp g^{1,0}(n) \setminus \cup_{m < n} \supp g^{1,0}(m)$ such that \[\left( \frac{g(n)(M_n)}{M_n !} : n \in \omega \right)\]

\noindent
is a 1-1 sequence that converges to some $ u \in  (\mathbb R\setminus \mathbb Q)\cup \{-\infty, 0, \infty\}.$
\\

We say that $g$ is of \emph{type 6} if $u=0$.\\

We say that $g$ is of \emph{type 7} if $u \in \mathbb R\setminus \mathbb Q$.\\

We say that $g$ is of \emph{type 8} if $u$ is $\infty$ or $-\infty$. \\

We say that $g$ is \emph{of type 9} if $\left\{ \dfrac{g(n)(M)}{M!} : M \in \supp g^{1,0}(n), n \in \omega \right\}$ is finite and $|\supp g^{1,0}(n)| > n$ for every $n \in \omega$.
\end{defin}

\begin{defin}[The types related to $\vec{P_0} \cup R_0$]
We define types 10 and 11 (with respect to $G$) as follows:
Let $g : \omega \to \mathbb X$ be such that $\supp g(n) \subseteq R_0 \cup (P_0 \times \omega)$ and $[g(n)]\in G$ for each $n\in \omega$ \\

We say that $g$ is \emph{of type 10} if $q(g^0(n)) > n$, for every $n \in \omega$.\\

We say that $g$ is \emph{of type 11} if  the family $\{[g(n)]: n \in \omega\}$ is an independent family whose elements have a fixed order $k$, for some positive integer $k$.
\end{defin}

Notice that if $G$ is a subgroup of $\mathbf X$ then  $\mathcal H_G=G^\omega \cap \mathcal H_{\mathbf X}$. We also set $\mathcal H=\mathcal H_{\mathbf X}$.

The following lemma is easy to verify and left to the reader.

\begin{lem}
Being a sequence of one of the types is absolute for transitive models of ZFC.
\end{lem}

 \section{The Partial order and the topology in the forcing extension}
In this section we define the forcing poset and prove its basic properties.
\subsection{Countable homomorphisms}
In this subsection, we will state a theorem that guarantees (in ZFC) the existence of partial homomorphisms defined on countable subgroups of $\mathbf{X}$. The statement of this theorem is very similar to Proposition 4.3. of \cite{bellini&boero&castro&rodrigues&tomita} and has a completely analogous proof, therefore we omit it. The only difference is that in \cite{bellini&boero&castro&rodrigues&tomita} the countable subset of $\mathcal H$ was not indexed and that proposition was originally stated for 1-1 indexations, but the proof is the same if one allows any sequence. 

\begin{prop} \label{countable.hom}
Let $E$ be a countable subset of $2^\mathfrak c$ containing $\omega$, $e \in \mathbf{X}_E$ with $e\neq 0$, a countable $\{g_k:\, k \in \omega\} \subseteq \mathcal H_{\mathbf{X}_E}$ and $A_k$ infinite subsets of $\omega$ for each $k \in \omega$.

Fix a family $(c_k: k \in \omega)$ of elements of $\mathbf {X}_E$ such that $[c_k] \in\mathbf X_E$,  $c_k$ is a non torsion element if $g_k$ is of one of types from 1 to 10, and $[c_k]$ has the same order as $[g_k]$ if $g_k$ is of type 11.

Then there exists a homomorphism $\rho:\, \mathbf{X}_{E} \to \mathbb T$ such that:

\begin{enumerate}
\item $\rho (e) \neq 0$,
\item for each $k \in \omega$, there exists $B_k\subseteq A_k$ infinite such that  $(\rho([g_k(n)]))_{n \in B_k}$ converges to $\rho([c_k])$, and
\item $\left(\rho\left(\frac{n!}{S}\chi_n\right): n \in \omega\right)$ converges to $0 \in\mathbb T$, for every integer $S>0$.
\end{enumerate}
\end{prop}

Now we are ready to define the forcing poset we are going to use.

\begin{defin} \label{the.order}
We define $\mathcal P$ as the set of the tuples of the form  $(E, \alpha, \phi, \mathcal G, c, A)$ such that:

\begin{itemize}
    \item $E$ is a countable subset of $2^\mathfrak c$ containing $\omega$,
    \item $\alpha < \mathfrak c$,
    \item $\mathcal G=(\mathcal G_{n, m}: \{(n, m): n\geq 1, m>1\})$ is such that each $\mathcal G_{n, m}$ is a countable subset of $\mathcal H$, where the types are defined with respect to $\mathbf{X}_E$. If $n=1$, the elements of $\mathcal G_{n, m}$ are sequences of types 1-10. If not, they are all of type $11$ and order $n$.
    \item $A=(A_{n,m,g}: n \geq 1, m>1, g \in \mathcal G_{n, m})$ is such that each $A_{n, m, g}$ is an infinite subset of $\omega$,
    \item $c=(c_{n,m,g}: n \geq 1, m>1, g \in \mathcal G_{n, m}$) is a family of elements of $\mathbf {X}_E$,
    \item if $n, m\geq 2$ and $g \in \mathcal G_{n, m}$, $c_{n,m,g} $ is an element of order $n$ with $c_{n,m,g}=[\frac{1}{n}\chi_{(\mu,0)}]$ for some $\mu \in C_{n,m} $,
    \item if $m\geq 2$ and $g \in \mathcal G_{1, m}$, $c_{1,m,g} =[ \chi_{(\mu)}] $ for some $\mu \in C_{1,m}$, 
    \item $\phi :\, \mathbf{X}_E \to \mathbb T^\alpha$ is an homomorphism,
    \item $(\phi([g(k)]))_{k \in A_{n,m,g}}$ converges to $\phi(c_{n,m,g})$ for each $ n\geq 1, m > 1$,
    \item $(\phi([\frac{1}{N}\chi_n)])_{n\in \omega}$ converges to $0 \in \mathbf X$, for every natural $N\geq 1$.
    \end{itemize}
    
We define $(E, \alpha, \phi, \mathcal G, c, A)\leq  (E', \alpha', \phi', \mathcal G', c', A')$
if:

\begin{enumerate}
    \item $E\supseteq E'$
    \item $\alpha \geq \alpha'$
    \item $\mathcal G_{n,m} \supseteq \mathcal G'_{n,m}$ for every $n\geq 1$ and $ m>1$
    \item $c_{n,m,g}=c'_{n,m,g}$ for each $n\geq 1$, $m>1$ and $g \in \mathcal G'_{n,m}$
    \item $A_{n,m,g} \subseteq ^* A'_{n,m,g}$ for each $n\geq 1$, $m>1$ and $g \in \mathcal G'_{n,m}$
    \item For every $\xi<\alpha'$ and $a \in X_{E'}$, $\phi(a)(\xi)= \phi'(a)(\xi)$.
\end{enumerate}
 Given $p \in \mathcal P$, we may denote its components by $E_p$, $\alpha^p$, $\phi^p$, $\mathcal G^p$, $c^p$ and $A^p$.

 If $G$ is a generic filter over $\mathcal P$ then the generic homomorphism defined by $G$ is the mapping $\Phi$ of domain $\bigcup\{\dom \phi^p: p \in G\}$ into $\mathbb T^\mathfrak c$ defined by $\phi(\cdot)(\xi)=\bigcup\{\phi^p(\cdot)(\alpha):p \in G\}$. In other words, if $p \in G$, $a \in \mathbf X_{E_p}$ and $\xi<\alpha_p$, then $\Phi(a)(\xi)=\phi^p(a)(\xi)$.
\end{defin}
 
Of course, we must see that the generic homomorphisms are really well defined homomorphisms into $\mathbb T^{\mathfrak c}$. We will see later that by assuming CH in the ground model, $\mathcal P$ is $\omega_1$ closed and has the $\omega_2$-c.c., therefore it preserves cardinals and $\mathfrak c$. We reserve the rest of this section to prove this fact.

 \begin{prop} Let $e \in \mathbf X$ be a non-zero element. Then 
 $\mathcal C_e=\{ p\in \mathcal P:\, e \in \mathbf{X}_{E^p}, \phi^p(e)\neq 0\}$ is open and dense in $\mathcal P$.\label{denseinjective}
 \end{prop}

 \begin{proof} Openness is clear. Fix $p \in \mathcal P$. We will define an extension $q \leq p$ that is an element of $\mathcal C_e$.
 
 Let $E^q = E^p \cup \supp e$ and $\alpha^q=\alpha^p+1$. Extend $\phi^p:\, \mathbf{X}_{E^p} \to \mathbb T^{\alpha_p}$ to
 a homomorphism $\phi:\, \mathbf{X}_{E_q} \to \mathbb T^{\alpha_p}$ using divisibility. Apply Proposition \ref{countable.hom} with  $\{ (n,m,g):\, g\in \mathcal G_{p,n,m}\}$, $\{ A_{n,m,g}^p:\, g\in \mathcal G^p_{n,m}\}$ and $\{ c_{n,m,g}^p:\, g\in \mathcal G_{n,m}^p\}$. Then there exists $\rho:\, \mathbf{X}_{E^q} \to \mathbb T$ such that 
 
 \begin{enumerate}
 \item $\rho(e) \neq 0$,
 
\item for each $(n,m,g)$ with $n\geq 1$, $m>1$ and $g \in \mathcal G_{n,m}^p$,
 there exists $B_{n,m,g} \subseteq A_{n,m,g}^p$ infinite such 
 $(\rho([g(k)]))_{k\in B_{n,m,g}}$ converges to $\rho([c_{n,m,g}^p])$ and
 
 \item $(\rho([\frac{n!}{S}\chi_n]):\, n \in \omega)$ converges to $0\in \mathbb T$, for every positive integer $S$.
 \end{enumerate}
 
 Set $\mathcal G_{n,m}^q =\mathcal G_{n,m}^p $ for each $n \geq 1$ and $m>1$. Set $c_{n,m,g}^q=c_{n,m,g}^p$, $A_{n,m,g}^q=B_{n,m,g}$ for each $g \in \mathcal G_{n,m}^p$ with $n \geq 1$, $m>1$ and $\phi^q= {\phi^p}^\frown \rho$. 
 
 Then $q\leq p$ and $q \in \mathcal C_e$.
 \end{proof}
 
 \begin{prop} For each $\alpha < \mathfrak c$ the set
 $\mathcal A_\alpha =\{ p \in \mathcal P:\, \alpha^p > \alpha\}$ is an open dense subset of $\mathcal P$. \label{densecodomain}
 \end{prop}
 
 \begin{proof} Fix $p \in \mathcal P$. If $\alpha < \alpha^p$ then $p \in \mathcal A_\alpha$. So suppose that $\alpha \geq \alpha^p$.
 
 We will define $q$. Set $E^q=E^p$, $\alpha_q=\alpha+1$, $\mathcal G_{n,m}^q=\mathcal G_{n,m}^p$, $c_{n,m,g}^q=c_{n,m,g}^p$, $A_{n,m,g}^q=A_{n,m,g}^p$. Let $\rho: \mathbf{X}_{E^p} \longrightarrow \{0\} ^{[\alpha^p, \alpha_q[}$ and $\phi^q={\phi^p}^\frown \rho$.
 
 Then $q\leq p$ and $q \in \mathcal A_\alpha$.
 \end{proof}
 
  \begin{prop} \label{omega1closed}The partial order $\mathcal P$ is $\omega_1$-closed.
 \end{prop}
 \begin{proof} Fix a decreasing sequence $(p_t:\, t<\omega)$. Write $p_t=(E^t, \alpha^t, \phi^t, \mathcal G^t, c^\xi, A^t)$. We define a common extension $r$ as follows:
 
 Let $E^r=\bigcup\{ E^{t}:\, t<\omega\}$,   $\mathcal G_{n,m}^r= \bigcup\{ \mathcal G_{n,m}^{t} :\, t \in \omega\}$ for each $n\geq 1$ and $m>1$.  For each $n\geq 1$, $m>1$ and $g \in \mathcal G_{n,m}^r$, define $c_{n,m,g}^r=c_{n,m,g}^{t}$ for some (every) $t$ such that $g \in \mathcal G_{n,m}^{t}$ (the value does not depend of $t$). Fix $A_{n,m,g}^{r}$ a pseudointersection of $\{A_{n,m,g}^{t}:\, g \in \mathcal G_{p_t, n,m} \}$.
 
 Let $\alpha^r=\sup\{\alpha^t: t<\omega\}$

Given $\xi<\alpha$ and $a \in \mathbf X_{E^r}=\bigcup_{t<\omega}\mathbf X_{E^t}$, let $\phi(a)(\xi)=\phi^t(a)(\xi)$ for some (every) $t$ such that $a \in X_{E^t}$ and $\alpha^t>\xi$.
 \end{proof}
 
 \begin{prop}The partial order $\mathcal P$ has the $\mathfrak c^+$-cc.\label{cpluscc}
 \end{prop}
 \begin{proof} Fix an arbitrary subset $\mathcal Q$ of $\mathcal P$ of cardinality $\mathfrak c^+$. We show that there $\mathcal Q$ has a subset of $\mathfrak c^+$-many pairwise compatible elements.
 
Fix $\mathcal Q_0\subseteq \mathcal Q$ of cardinality $\mathfrak c^+$ and $\alpha <\mathfrak c$ such that $\alpha^p=\alpha^q$ for every $p, q \in \mathcal Q_0$.
 
 Using the $\Delta$-system Lemma, there exists 
 $\mathcal Q_1\subseteq \mathcal Q_0$ of cardinality $\mathfrak c^+$ such that $\{E^p:p \in \mathcal Q_1\}$ is a $\Delta$-system of root $\tilde E$. Furthermore, using the fact that $\tilde{E}^\omega$
 has cardinality at most $\mathfrak c$, it follows that there exists $\mathcal Q_2\subseteq \mathcal Q_1$ of cardinality $\mathfrak c^+$ such that $\phi^p|_{{\mathbf X}_{\tilde E}}=\phi^q|_{{\mathbf X}_{\tilde E}}$ for every $p, q \in \mathcal Q_2$.
 
 For each $p \in \mathcal Q_2$, let  $J^p=\{(n,m,g):\, n, m \geq 1, g \in \mathcal G_{n,m}^p\}$ for each $p \in \mathcal Q_2$. Using the $\Delta$-system Lemma, we can find $\mathcal Q_3\subseteq \mathcal Q_2$ of cardinality $\mathfrak c^+$ such that  $\{J^p:\, p \in \mathcal Q_3\}$ is a delta system of root $\tilde J$.
 
 Notice that $\mathbf{X}_{\tilde E}^{\tilde J}$ has cardinality $\mathfrak c$, so there exists $\mathcal Q_4\subseteq \mathcal Q_3$ of cardinality $\mathfrak c^+$ such that for every $p, q \in \mathcal Q_4$ and $(n, m, g) \in \tilde J=J^p\cap J^q$, $c^p_{n, m, g}=c^q_{n, m, g}$. Similarly, since
 ${([\omega]^\omega)}^{\tilde J}$ has cardinality $\mathfrak c$, there exists $\mathcal Q_5\subseteq \mathcal Q_4$ of cardinality $\mathfrak c^+$ such that for every $p, q \in \mathcal Q_5$ and $(n, m, g) \in \tilde J=J^p\cap J^q$, $A^p_{n, m, g}=A^q_{n, m, g}$.
 
Given $p, q \in \mathcal Q_5$, a common extension is given by the element $r$ whose components are defined as follows: $E^r=E^p\cup E^q$, $\alpha^r=\alpha^q=\alpha^p$, $\mathcal G_{n, m}^r=\mathcal G_{n, m}^p\cup \mathcal G_{n, m}^q$, $A_{n, m, g}^r=A_{n, m, g}^s$ and $c^r_{n, m, g}=c^s_{n, m, g}$ if $(n, m, g) \in J^s$ (where $s \in \{p, q\}$).
 
 To define $\phi^r$, notice that $\mathbf X_{E^p\setminus \tilde E}\oplus \mathbf X_{\tilde E}\oplus \mathbf X_{E^q\setminus \tilde E}$. Let $\pi_0:\mathbf X_{E^r}\to \mathbf X_{E^p\setminus \tilde E}$, $\pi_1:\mathbf X_{E^r}\to \mathbf X_{\tilde E}$, $\pi_2:\mathbf X_{E^r}\to \mathbf X_{E^q\setminus \tilde E}$ be the projections. Define $\phi^r=\phi^p\circ \pi_0+\phi^p\circ \pi_1+\phi^q\circ \pi_2=\phi^p\circ \pi_0+\phi^q\circ \pi_1+\phi^q\circ \pi_2$.
 \end{proof}
 
 \begin{prop}\label{typedense} Let $g$ be sequence of one of the types of $\mathbf X$ and $m> 1$. If $g$ is of types 1 to 10, let $n=1$. If $g$ is type 11, let or $n$ the order of $g$.
 Then $\mathcal S_{n,m,g}= \{p\in \mathcal P:\, g \in \mathcal G_{n,m}^p\}$ is open and dense in $\mathcal P$.
 \end{prop}
 
 \begin{proof} Let $p \in \mathcal P$ be an arbitrary condition. Fix $E$ countable such that $E^p \subseteq E$ and
 $g(k) \in \mathbf{X}_E$ for each $k \in \omega$.
 
 Fix $\mu \in C_{n,m} \setminus E$. We set $E^q=E \cup \{\mu\}$. For each $(m',n') \neq (m,n)$ with $m',n' \geq 1$,
 define $\mathcal G_{n',m'}^q=\mathcal G_{n',m'}^p$ and $\mathcal G_{n,m}^q=\mathcal G_{n,m}^p \cup \{g\}$. Set $\alpha^q=\alpha^p$.
 
 For every  $n'\geq 1$, $m'>1$ and $g' \in \mathcal G_{p,n',m'}\setminus\{g\}$,
 define $c_{n',m',g'}^q=c_{n',m',g'}^p$ and $A_{n',m',g'}^q=A_{n',m',g'}^p$. It remains to define $c_{n,m,g}^q$ and $A_{n,m,g}^q$. Let $c_{n,m,g}^q= \frac{1}{n}\chi_{(\mu,0)}$ if $n>1$ and $\chi_\mu$ if $n=1$.
 
 Extend $\phi_p:\, \mathbf{X}_{E^p} \to \mathbb T^{\alpha^p}$ to $\phi:\, \mathbf{X}_{E}\to  \mathbb T^{\alpha^p}$ using divisibility. Now, let $A \subseteq \omega$ be an infinite such that the sequence $(\phi([g(k)]):\, k \in A)$ is convergent, as $\mathbb T^{\alpha^p}$ is a compact metric space. Extend $\phi$ to a homomorphism 
 $\phi^q: \mathbf{X}_{E^q} \to \mathbb T^{\alpha^p}$ such that $\phi^q([c_{q,n,m,g}])= \lim (\phi([g(k)]):\, k \in A)$. Set $A_{n,m,g}^q=A$. Then $q\leq p$ and $q \in \mathcal S_{n,m,g}$.
 \end{proof}

 \begin{teo} \label{main.forcing}
Assume CH in the ground model $\mathbf V$. Then $\mathcal P$ preserves cardinals, $\mathfrak c$ and does not add reals. If $G$ is generic over $\mathcal P$, then the $G$-generic homomorphism $\Phi$ is a well defined injective homomorphism from $\mathbf X$ into $\mathbb T^{\mathfrak c}$. Moreover, the following holds:
\begin{enumerate}
    \item For every sequence $g$ of one of the types from $1$ to $10$ in $\mathbf X$ and $m\geq 1$, there exists $\mu \in C_{1,m}$ such that $[\chi_{\mu}]$ is an accumulation point of $(\Phi([g(k)]):\, k \in \omega)$
    \item For every sequence $g$ of type 11 and order $n$ in $\mathbf X$ and for every $m\geq 1$, there exists $\mu \in C_{n,m}$ such that $[\frac{1}{n} \chi_{\mu,0}] $  is an accumulation point of $(\phi([g(k)]):\, k \in \omega)$. 
    \item $(\Phi([\frac{1}{N}\chi_n)])_{n\in \omega}$ converges to $0 \in \mathbf X$, for every natural $N\geq 1$.
\end{enumerate}
    
\end{teo}

\begin{proof}
 By CH, propositions \ref{omega1closed} and \ref{cpluscc}, $\mathcal P$ is $\omega_1$ closed and has the $\omega_2$ chain condition, so $\mathcal P$ preserves cardinals, does not add reals and preserves $\mathfrak c$. Notice that since being a type is absolute for transitive models of ZFC, the functions of type 1 to 11 are the same in the ground model and in the extension.
 
 Let $G$ be a $\mathcal P$-generic filter and $\Phi$ the associated generic homomorphism. 
 
 $\Phi$ is well defined: suppose $p, q \in G$, $\xi<\alpha^p\cap \alpha^q$ and $e \in \mathbf X_{E^p}\cap \mathbf X_{E^q}$. We must see that $\phi^p(e)(\xi)=\phi^q(e)(\xi)$. Since $G$ is a filter, there exists $r$ such that $r\leq p, q$, so $\xi<\alpha^r$, $e \in \mathbf X_{E^r}$ and $\phi^p(e)(\xi)=\phi^r(e)(\xi)=\phi^q(e)(\xi)$

Now we verify that the domain of $\Phi$ is $\mathbf X$, that the codomain is $\mathbb T^\mathfrak c$ and that $\Phi$ is injective at the same time. It is clearly that the domain contained in $\mathbf X$. Let $e\neq 0$ be an element of $X$ and $\alpha<\mathfrak c$. By propositions \ref{denseinjective} and \ref{densecodomain},
$\mathcal C_e$ and $\mathcal A_\alpha$ are open and dense subsets of $\mathcal P$, therefore there exists $p \in G$ such that $\alpha^p>\alpha$ $e \in \mathbf X_{E^p}$, $\phi^p(e)\neq 0$. So there exists $\xi<\alpha^p$ such that $\phi^p(e)(\xi)\neq 0$, which implies that $\Phi(e)(\xi)\neq 0$. Moreover, $\alpha \subseteq \dom \Phi(e)\subseteq \mathfrak c$. Since $\alpha$ is arbitrary, $\dom \Phi(e)=\mathfrak c$.

$\Phi$ is an homomorphism: given $e, e' \in \mathbf X$, by \ref{denseinjective} there exists $p \in G$ such that $e, e', e+e' \in \mathbf X_{E^p}$. Since $\phi^p$ is an homomorphism, it follows that $\Phi(e+e')=\phi^p(e+e')=\phi^p(e)+\phi^p(e')=\Phi(e)+\Phi(e')$.

 Let $g$ be a type and $m> 1$. If $g$ is of type 1 to 10, let $n=1$. If  $g$ is type 11, let $n$ be the order of $g$.
 Then by Proposition \ref{typedense}, there exists $G \cap \mathcal S_{n,m,g}$. Fix $p$ in this intersection. We claim $\Phi(c_{p,n,m,g})$ is an accumulation point of $\Phi([g])$.
 
 We know $\phi_p (c_{n,m,g}^p)$ is the limit of the convergent sequence $(\phi_p([g(k)]):\, k\in A_{n,m,g}^p)$. Fix $F$ a finite subset of $\mathfrak c$ and let $\alpha $ such that $F$ is a subset of $\alpha$. Let $q \leq p$ such that $\alpha < \alpha^q$ (which exists since $\mathcal A_{\alpha +1}  \cap G\neq \emptyset$). Then $( \pi_F \circ \Phi ([g(k)]):\, k \in A_{n,m,g}^q)$ converges to $\pi_F \circ \Phi(c_{n,m,g}^q)$. Since $c_{n,m,g}^q=c_{p,n,m,g}^q$, this concludes that 
 $\Phi(c_{n,m,g}^p)$ is an accumulation point of $(\Phi([g(k)]):\, k \in \omega )$.
 
It remains to see that $(\Phi([\frac{n!}{S}\chi_n]):\, n \in \omega)$ is a convergent sequence in $0 \in \mathbb T^{\omega_1}$. Let $\xi<\mathfrak c$.  Let $p\in G$ such that $\alpha^p >\xi$. Then  $(\pi_\xi \circ \Phi([\frac{n!}{S}\chi_n]):\, n \in \omega)= (\pi_\xi \circ \phi_p([\frac{n!}{S}\chi_n]):\, n \in \omega)$ converges to $0$. Since $\xi$ is arbitrary, we are done.
 \end{proof}

 \subsection{The subspace topology on large subgroups of \texorpdfstring{$\mathbf X$}{X}}
 
 Of course, not every subgroup of $\mathbf X$ is countably compact with the forced topology. However, some of them are if they have enough accumulation points. Thus, we define the concept of \textit{large subgroup} of $\mathbf X$.

\begin{defin} \label{large.subgroup}
Let $H$ be a subgroup of $\mathbf X$. Let $D$ be the set of all integers $n > 1$ such that $H$ contains an isomorphic copy of the group $\mathbb{Z}_{n}^{(\mathfrak{c})}$.

We say that $H$ is a \textit{large subgroup of} $\mathbf X$ if $2^{\mathfrak c}\geq |G| \geq r(G)\geq \mathfrak c$, for all $d, n \in \mathbb{N}$ with $d \mid n$ the group $dG[n]$ is either finite or has cardinality at least $\mathfrak{c}$ and there exist $(k_n:\, n \in D)$  with $k_n$ a positive integer such that:

\begin{enumerate}[label=\roman*)]
\item $\{ \chi_{\mu} \in \mathbf{X} : \mu \in  C_1\} \subseteq H$, and
\item $\{[\frac{1}{n}\chi_{(\mu,0)}] \in \mathbf{X} : n \in D,\,\mu \in \cup_{n \in D} C_{n,k_n}\} \subseteq H$.
\end{enumerate}
\end{defin}

\begin{teo} \label{main.theorem} Consider $\mathbf X$ with the group topology in Theorem \ref{main.forcing}. If $H$ is a large subgroup of $\mathbf X$, then it is countably compact in the subspace topology and has convergent sequences.
\end{teo}

\begin{proof}It follows from Theorem \ref{main.forcing} that if $S$ is a positive integer, then the sequence $\left ( \frac{1}{S}   n!   \chi_{n} : n \in \omega \right)$ converges to the neutral element of $\mathbf X$. Since both the elements of the sequence is eventually in $H$ and the limit is in $H$, it follows that $H$ has non-trivial convergent sequences.

Let $g: \omega \to H$. Take any $\tilde g:\omega\to\mathbb X$ such that $[\tilde g]= g.$ It follows from Proposition \ref{prop_subseq} that there exist $h:\omega\rightarrow \mathbb X$ such that $h\in \mathcal{H}_H$ or $[h]$ is a constant in $H$, $c\in \mathbb X$ with $[c] \in H$, $F \in [\omega]^{< \omega}$, $p_i, q_i \in \mathbb{Z}$ with $q_i \neq 0$ for every $i \in F$, $(j_{i}:i\in F)$ increasing enumerations of subsets of $\omega$  and $j: \omega \to \omega$ strictly increasing such that

\[\tilde g \circ j = h + c + \sum_{i \in F} \frac{p_i}{q_i}   f\circ j_{i}\]

with $q_i\leq j_i(n)$ for each $n \in \omega$ and $i \in F$, where $f : \omega \to \mathbb X$ is given by $f(n) = n!\chi_n$ for every $n \in \omega$.

In the case where $[h]$ is constant, say constantly $v\in\mathbf{X}$, we have that $g\circ j=[\tilde g]\circ j=[\tilde g\circ j]=v+[c]+\sum_{i\in F}[\frac{p_i}{q_i}f\circ j_i])$ converges to $v+[c]$.

In the case $h\in\mathcal H _H$ and $h$ is type 11 of order $n$, then $H$ contain infinitely many copies of $\mathbb Z_n$. Thus, by hypothesis, $n \in D$. Since $n\in D$, it follows from $\mathcal H_H \subseteq \mathcal H$ that $(h(k):\, k \in \omega)$ has an accumulation point $[\frac{1}{n}\chi_{\mu,0}]$ with $\mu \in C_{n,k_n}$. Hence, an accumulation point of $h$ in $H$. Thus the sequence $(g\circ j(k):\, k \in \omega)$
has an accumulation point in $[\frac{1}{n}\chi_{\mu,0}] +[c]$ in $H$.

In the case $h\in\mathcal H _G$ and $h$ is type 1 to 10, it follows from $\mathcal H_G \subseteq \mathcal H$ that $(h(k):\, k \in \omega)$ has an accumulation point  $[\chi_{\mu}]$ with $\mu \in C_{1}$. Hence, an accumulation point of $h$ in $H$. Thus the sequence $(g\circ j(k):\, k \in \omega)$ has accumulation point $[\chi_{\mu}]+[c]$ in $H$.
\end{proof}

\section{The classification of Abelian groups of cardinality \texorpdfstring{$2^\mathfrak c$}{2c}.}

\subsection{Immersions} 

We change slightly the statement and the notation of Proposition 6.1. in \cite{bellini&boero&castro&rodrigues&tomita} to facilitate the application, but it is implicit in the proof in
\cite{bellini&boero&castro&rodrigues&tomita}. 

We define $\mathbf {A}=(\mathbb Q /\mathbb Z)^{(P_0)} \oplus \mathbb Q ^{(P_1)}\oplus \mathbf U$.

\begin{defin} We say that $\mathbf W$ is a nice subgroup of $\mathbb W_\mathfrak c$ if there exists a family of positive integers $(n_\xi)_{\xi \in P_0}$ such that 

$\mathbf W=(\mathbb{Q} /\mathbb{Z})^{(\bigcup_{\xi \in P_0 } \{\xi\}\times n_\xi)} \oplus\mathbb{Q}^{(P_1)}$. For this $\mathbf W$ we denote $\vec{P_0}= (\bigcup_{\xi \in P_0 } \{\xi\}\times n_\xi )$, so $\mathbf W=(\mathbb{Q} /\mathbb{Z})^{(\vec{P_0})} \oplus\mathbb{Q}^{(P_1)}$.
\end{defin}
\begin{prop} \label{megazord}
Let  $H$ be an Abelian group  such that $|H| = r(H) = \mathfrak{c}$ with $H$ a subgroup of $\mathbf{A}_\mathfrak c$.

Let $D$ be the set of all integers $n > 1$ such that $H$ contains an isomorphic copy of the group $\mathbb{Z}_{n}^{(\mathfrak{c})}$.

Then there exist $\mathbf{W}$ a nice subgroup of $\mathbb W_\mathfrak c$, $K_1 \in [P_1]^{\mathfrak c}$ with $\omega \subseteq K_1$, a family $(K_n: n \in D)$ of pairwise disjoint elements of $[P_0]^{\mathfrak c}$, a family $(z_\xi: \xi \in \bigcup_{n \in D} K_n)$ and a group monomorphism $\phi: \mathbf{ A}_\mathfrak c \to \mathbf{W}$ such that:

\begin{enumerate}[label=\alph*)]
\item $\{ \chi_{\xi} \in \mathbf{W}_{P_1} : \xi \in K_1\} \subseteq \phi[H]$,
\item $\{z_{\xi} \in \mathbf{W}_{P_0} : \xi \in \cup_{n \in D} K_n\} \subseteq \phi[H]$,
\item $\ordem(z_{\xi}) = n,\ \forall \xi \in K_n, n \in D$ and
\item $\supp z_{\xi} \subseteq \{\xi\} \times \omega \ \forall \xi \in \cup_{n \in D} K_n$.
\end{enumerate}

We say that $\phi$ is a nice immersion for $H$. \qed
\end{prop}

This proposition follows from Proposition 6.1. of \cite{bellini&boero&castro&rodrigues&tomita}: $\mathbb W_{\mathfrak c}$ is naturally isomorphic to the group $\mathbb W$ that appears in \cite{bellini&boero&castro&rodrigues&tomita}, and this natural isomorphism preserves the nice subgroups defined in that article. Finally, $\mathbf W$ is divisible, so we can extend the isomorphism that Proposition 6.1. gives us to the whole group $\mathbf A_{\mathfrak c}$.

\begin{prop} \label{megazord.returns}
Let  $H$ be an Abelian group  such that $2^\mathfrak c \geq |H| \geq r(H) \geq \mathfrak{h}$.

Let $D$ be the set of all integers $n > 1$ such that $H$ contains an isomorphic copy of the group $\mathbb{Z}_{n}^{(\mathfrak{c})}$.

Then there exists a family $(k_n:\, n \in D)$ of positive integers and a group monomorphism $\varphi: H \to \mathbf{X}$ such that:

\begin{enumerate}[label=\roman*)]
\item $\{ \chi_{\mu} \in \mathbf{X}_{P_1 \cup R_1} : \mu \in  C_1\} \subseteq \varphi[H]$, and
\item $\{[\frac{1}{n}\chi_{(\mu,0)}] \in \mathbf{X}_{P_0\cup R_0} : n \in D,\,\mu \in \cup_{n \in D} C_{n,k_n}\} \subseteq \varphi[H]$.
\end{enumerate}

Thus, $G$ is isomorphic to a large subgroup of $\mathbf X$.
\end{prop}

\begin{proof} By theorems \ref{teo_classificacao_grupos_divisiveis} and \ref{teo_imersao_num_grupo_divisivel}, we may consider $H$ is a subgroup of $\mathbb A$. Then we can fix a subgroup $\tilde H$ of $H$ of cardinality $\mathfrak c$ such that $r(\tilde H)=\mathfrak c$ and for every $n \in D$, there exists a copy of $(\mathbb {\mathbb Z}_n)^\mathfrak c$ in $\tilde H$. By a trivial permutation  of coordinates we can assume that $\tilde H$ is a subgroup $\mathbb A_{\mathfrak c}$.
 Applying Proposition \ref{megazord}, there exist $\mathbf{W}$ a nice subgroup of $\mathbb W_\mathfrak c$, $K_1 \in [P_1]^{\mathfrak c}$ with $\omega \subseteq K_1$, a family $(K_n: n \in D)$ of pairwise disjoint elements of $[P_0]^{\mathfrak c}$, a family $(z_\xi: \xi \in \bigcup_{n \in D} K_n)$ and a group monomorphism $\phi: \mathbb A_\mathfrak c \to \mathbf{W}_\mathfrak c= \mathbf{W}_{P_0}\oplus \mathbf{W}_{P_1}$ such that:

\begin{enumerate}[label=\alph*)]

\item $\{(\chi_{\xi} \in \mathbf{W}_{P_1} : \xi \in K_1\} \subseteq \phi[\tilde H]$,
\item $\{z_{\xi}\in \mathbf{W}_{P_0} : \xi \in \cup_{n \in D} K_n\} \subseteq \phi[\tilde H]$,
\item $\ordem(z_{\xi}) = n, \ \forall \xi \in K_n, n \in D$ and
\item $\supp z_{\xi} \subseteq \{\xi\} \times \omega, \ \forall \xi \in \cup_{n \in D} K_n$.

We can shrink $K_n$ if necessary to find $k_n$ positive integer such that \\

\item $|\supp z_\xi|= k_n$, for each $n \in D$ and $\xi \in K_n$ .\\

By making some permutation within each $ \{\xi\} \times k_n$
we can further assume that \\

\item $\supp z_\xi =\{\xi\} \times k_n$ for each $n \in D$ and $\xi \in K_n$.

\end{enumerate}

Define $\mathbf W = \mathbf{W}_\mathfrak c \oplus \mathbf{U} $.

We can assume that $\phi:\, \mathbb A_\mathfrak c \to \mathbf W$  and extend it to $\phi:\, \mathbb A= \mathbb A_\mathfrak c \oplus \mathbf{U} \to \mathbf{W} =\mathbf{W}_\mathfrak c \oplus \mathbf{U}$, using the identity on $\mathbf{U}$.

Fix  $\sigma_n$ a bijection between $K_n$  and $C_{n,k_n}$ for each $n\in D$ and $\sigma_1$  a bijection between $K_1$ and $C_1$ with $\sigma_1(k)=k$ for every $k\in \omega$.\\

Define $\eta:\, \mathbf W \longrightarrow \mathbf X$ an injective homomorphism such that

- $\eta:\,\mathbf{W}_{\{\xi\}\times k_n} \to \mathbf{X}_{\{\sigma_n(\xi)\}\times k_n}$ is an isomorphism with
$\eta(z_\xi)=[\frac{1}{n}\chi_{(\sigma_n(\xi),0)}]$ for each $n \in D$ and $\xi \in K_n$ (this is possible by condition $f)$), 

- $\eta([\chi_{\xi}])= [\chi_{\sigma_1(\xi)}]$ for each $\xi \in K_1$,

- $\eta$ restricted to  $\mathbf{U}$ is the identity.\\

Now, let $\varphi= \eta \circ \phi|_G$. The homomorphism is an embedding, since both $\eta$ an $\phi$ are injective homomorphisms.

Applying $\eta$ in $a)$ it follows that 
 $\{\eta( \chi_{\xi}) \in \eta[\mathbf{W}_\mathfrak c] \subseteq \eta[\mathbf W]: \xi \in K_1\} \subseteq \varphi[\tilde H]$.
 
 Therefore, $\{\chi_\mu \in \mathbf X:\, \mu \in C_1\} \subseteq \varphi[\tilde H] \subseteq \varphi[H]$ and $i)$ holds.
 
 Likewise, condition $ii)$ holds.
\end{proof}

\subsection{The classification}

\begin{teo} \label{theor.with.immersion} Consider $\mathbf X$ with the topology from Theorem \ref{main.forcing}.  Let $H$ be a group such that $2^\mathfrak c \geq |H| \ge r(H) = \mathfrak{c}$ and for all $d, n \in \mathbb{N}$ with $d \mid n$, the group $dG[n]$ is either finite or has cardinality at least $\mathfrak{c}$. Then $H$ admits a countably compact group topology with a non-trivial convergent sequence.
\end{teo}

\begin{proof} By Proposition \ref{megazord.returns}, the group $H$ is isomorphic to a large subgroup of $\mathbf X$ therefore by Theorem \ref{main.theorem}, $H$ admits a countably compact group topology with a convergent sequence.
\end{proof}

We can now (consistently) answer Dikranjan and Shakhmatov's question for Abelian groups of cardinality $\leq 2^{\mathfrak c}$.

\begin{cor}\label{main.corol} Consider the forcing model in Theorem \ref{main.forcing}

Let $H$ be a non-torsion Abelian group of size at most $2^\mathfrak c$. Then the following are equivalent

1) $(\star)$ $2^\mathfrak c \geq |H| \ge r(H) = \mathfrak{c}$ and for all $d, n \in \mathbb{N}$ with $d \mid n$, the group $dH[n]$ is either finite or has cardinality at least $\mathfrak{c}$;

2) $H$ admits a countably compact Hausdorff group topology 

3) $H$ admits a countably compact Hausdorff group topology with non-trivial convergent sequences. 
\end{cor}

\begin{proof}

If $H$ admits a countably compact group topology then condition $(\star)$ is satisfied in ZFC. Then $2)$ implies $1)$.

By Theorem \ref{theor.with.immersion}, if $H$ satisfies 
$(\star)$ then $3)$ holds.

Finally $3)$ implies $2)$ is obvious.
\end{proof}

\section{Some comments and questions}

 So far, the forcings that have been used in constructions of countably compact groups with some particular properties were the Cohen model, the Random model \cite{szeptycki&tomita} for groups of cardinality $\mathfrak c$ and a variation of forcing to construct Kurepa trees \cite{koszmider&tomita&watson}, \cite{tomita4}, \cite{tomita5}, \cite{castro-pereira&tomita} and \cite{dikranjan&shakhmatov2} for groups of cardinality at most $2^\mathfrak c$.

 The use of selective ultrafilters for countably compact groups with some property started for groups of cardinality $\mathfrak c$ (\cite{tomita&watson},\cite{garcia-ferreira&tomita&watson}), then for groups of cardinality $2^\mathfrak c$
 (\cite{madariaga-garcia&tomita}, \cite{tomita3}). To obtain larger examples, we have used $p$ compactness and the existence of basis in the ultrapower of a direct sum of torsion groups of bounded order \cite{castro-pereira&tomita2010} or the direct sum of $\mathbb Q$'s \cite{bellini&rodrigues&tomitaRationalStack}.
 
 Apparently, selective ultrafilters are not enough to produce a classification of countably compact groups of cardinality $2^\mathfrak c$. Thus, it seems that the best chance to classify the large Abelian groups that admit a countably compact group topologies is through some other type of forcing.
 
 \begin{prob} What other types of forcing can generate countably compact groups with some property of interest?
 \end{prob}
 
 Further questions in ZFC that are still open after the breakthrough in  M. Hrusak, U. A. Ramos-Garcia, J. van Mill and S. Shelah \cite{michaelnew} are:

 \begin{prob} Is there a countably compact group topology on the free Abelian group  of cardinality $\mathfrak c$ in ZFC (with/without non-trivial convergent sequEnces)? 
 \end{prob}
 
 Tomita \cite{tomita2019} showed in ZFC that if there exists a non-torsion countably compact Abelian group without non-trivial convergent sequences then there exists a countably compact free Abelian group without non-trivial convergent sequences.

 As it was the case for the classification obtained by Dikranjan and Tkacheko \cite{dikranjan&tkachenko}, an answer to the question above is the first step for the following:
 
  \begin{prob} Classify in ZFC the Abelian groups of cardinality $\mathfrak c$ that admit a countably compact group topology. 
 \end{prob}
 
 The existence of a countably compact free Abelian group of cardinality $\mathfrak c$ without non-trivial convergent sequences implies the existence of Wallace semigroups (a countably compact both-sided cancellative semigroup that is not algebraically a group). From the existence of a countably compact group topology without non-trivial convergent sequences in ZFC, it is natural to try to answer the following question 
due to Wallace:
 
 \begin{prob}
 Is there a Wallace semigroup  in ZFC?
 \end{prob}
 
 The known examples of Wallace semigroups are under CH \cite{robbie&svetlichny}, Martin's Axiom for countable posets \cite{tomita}, $\mathfrak c$ incomparable selective ultrafilter \cite{madariaga-garcia&tomita} and a single selective ultrafilter \cite{boero&castro&tomita2019}. Only the one in \cite{tomita} was not obtained as a semigroup of a countably compact free Abelian group without non-trivial convergent sequences.


\bibliographystyle{amsplain}
\bibliography{gruposenumeravelmentecompactos}

\end{document}